\documentclass[12pt]{amsart}
\usepackage[a4paper, top=1.25in, bottom=1.25in, right=1in,  left=1in, footskip=0.5in, headheight = 0.5in]{geometry}
\usepackage[foot]{amsaddr}
\usepackage[leqno]{amsmath}
\usepackage{thmtools, thm-restate}
\usepackage{thm-autoref}
\usepackage[T1]{fontenc}
\usepackage[utf8]{inputenc}
\usepackage[shortlabels]{enumitem}
\usepackage[center]{caption}
\usepackage{amsfonts}
\usepackage{tcolorbox}
\usepackage{graphicx}
\usepackage{caption}
\usepackage{todonotes}
\usepackage{pdfpages}
\usepackage{amssymb, amsthm,latexsym, mathdots, mathtools, graphicx, xkeyval, tikz, tkz-euclide, adjustbox, float, array, subcaption, listings, comment}
\usepackage[colorlinks=true,
linkcolor =black,
citecolor=red!50!black,
urlcolor=red!50!black]{hyperref}
\usepackage{cleveref}
\usepackage{xcolor}

\newtheorem{theorem}{Theorem}
\newtheorem{lemma}[theorem]{Lemma}

\newtheorem{question}[theorem]{Question}

\newcounter{tbox}
\newcommand{\clm}[1]{\vspace{0.1cm}\medskip\refstepcounter{tbox}\noindent{\parbox{\textwidth}{(\thetbox) \emph{#1}}}\vspace{0.2cm}}
\setlist[itemize]{leftmargin=5.5mm}

\usepackage{algorithm}
\usepackage{algpseudocode}

\newcommand*{\myproofname}{Proof}
\newenvironment{claimproof}[1][\myproofname]{\begin{proof}[#1]}{\end{proof}}

\usetikzlibrary{decorations.pathmorphing, decorations.pathreplacing, decorations.shapes}
\usetikzlibrary{decorations.markings}
\usetikzlibrary{calc}
\usetikzlibrary{shapes.geometric}

\pgfkeys{/tikz/.cd,
    contour distance/.store in=\ContourDistance,
    contour distance=0pt, 
    contour step/.store in=\ContourStep,
    contour step=1pt,
}

\pgfdeclaredecoration{closed contour}{initial}
{%
\state{initial}[width=\ContourStep,next state=cont] {
    \pgfmoveto{\pgfpoint{\ContourStep}{\ContourDistance}}
    \pgfcoordinate{first}{\pgfpoint{\ContourStep}{\ContourDistance}}
    \pgfpathlineto{\pgfpoint{0.3\pgflinewidth}{\ContourDistance}}
    \pgfcoordinate{lastup}{\pgfpoint{1pt}{\ContourDistance}}
    
  }
  \state{cont}[width=\ContourStep]{
     \pgfmoveto{\pgfpointanchor{lastup}{center}}
     \pgfpathlineto{\pgfpoint{\ContourStep}{\ContourDistance}}
     \pgfcoordinate{lastup}{\pgfpoint{\ContourStep}{\ContourDistance}}
  }
  \state{final}[width=\ContourStep]
  { 
    \pgfmoveto{\pgfpointanchor{lastup}{center}}
    \pgfpathlineto{\pgfpointanchor{first}{center}}
  }
}

\definecolor{Lavender}{rgb}{0.9, 0.9, 0.98}

\tikzset{
  cir/.style = {circle,draw,fill,inner sep=.7pt},
  circ/.style = {circle,draw,fill,inner sep=1.3pt},
  circb/.style = {circle,draw,inner sep=1.7pt},
  circg/.style = {circle,draw=gray,fill=gray,inner sep=1.3pt},
  circr/.style = {circle,draw=Crimson,fill=Crimson,inner sep=1.3pt},
  invisible/.style = {circle,draw=none,inner sep=0pt,font=\tiny},
  nonedge/.style={decorate,decoration={snake,amplitude=.3mm,segment length=6mm}},
}

\tcbuselibrary{skins}

\newtcolorbox{mybox}[1]{minipage boxed title*=-2cm,
enhanced,attach boxed title to top center=
{yshift=-3mm,yshifttext=-1mm},colback=Lavender!30!white,
boxed title style={size=small,colback=Lavender},coltitle=black,
center title,title={#1}}

\title[Maximum List $r$-Colorable Induced Subgraphs in $kP_3$-free Graphs]{Maximum List $r$-Colorable Induced Subgraphs in $kP_3$-free Graphs}
\author{Esther Galby$^{\dagger}$}
\address{$\dagger$ Chalmers University of Technology, Sweden.}
\author{Paloma T. Lima$^{\ddagger}$}
\address{$\ddagger$ IT University of Copenhagen, Denmark.}
\author{Andrea Munaro$^{\mathsection}$}
\author{Amir Nikabadi$^{\ddagger}$}
\address{$\mathsection$ University of Parma, Italy.}
\address{Lima and Nikabadi acknowledge the support of the Independent Research Fund Denmark grant agreement number 2098-00012B.}

\allowdisplaybreaks
\begin{document}
\maketitle
\begin{abstract}
We show that, for every fixed positive integers $r$ and $k$, \textsc{Max-Weight List $r$-Colorable Induced Subgraph} admits a polynomial-time algorithm on $kP_3$-free graphs. This problem is a common generalization of \textsc{Max-Weight Independent Set}, \textsc{Odd Cycle Transversal} and \textsc{List $r$-Coloring}, among others. Our result has several consequences. 

First, it implies that, for every fixed $r \geq 5$, assuming $\mathsf{P}\neq \mathsf{NP}$, \textsc{Max-Weight List $r$-Colorable Induced Subgraph} is polynomial-time solvable on $H$-free graphs if and only if $H$ is an induced subgraph of either $kP_3$ or $P_5+kP_1$, for some $k \geq 1$. Second, it makes considerable progress toward a complexity dichotomy for \textsc{Odd Cycle Transversal} on $H$-free graphs, allowing to answer a question of Agrawal, Lima, Lokshtanov, Rz{\k{a}}{\.z}ewski, Saurabh, and Sharma [TALG 2024]. Third, it gives a short and self-contained proof of the known result of Chudnovsky, Hajebi, and Spirkl [Combinatorica 2024] that \textsc{List $r$-Coloring} on $kP_3$-free graphs is polynomial-time solvable for every fixed $r$ and $k$.

We also consider two natural distance-$d$ generalizations of  
\textsc{Max-Weight Independent Set} and \textsc{List $r$-Coloring} and provide polynomial-time algorithms on $kP_3$-free graphs for every fixed integers $r$, $k$, and $d \geq 6$.
\end{abstract}


\section{Introduction}\label{sec:intro}
A fundamental class of graph optimization problems consists in finding a maximum-weight induced subgraph satisfying a certain fixed property $\Pi$. Lewis and Yannakakis~\cite{LY80} showed that, whenever this fixed property $\Pi$ is nontrivial and hereditary\footnote{A graph property $\Pi$ is \textit{nontrivial} if it is true for infinitely many graphs and false for infinitely many graphs. It is \textit{hereditary} if, whenever a graph satisfies $\Pi$, all its induced subgraphs satisfy $\Pi$ as well.}, the corresponding problem is $\mathsf{NP}$-hard. In this paper, we investigate the case where $\Pi$ is the property of being list $r$-colorable, leading to a meta-problem called \textsc{Max-Weight List $r$-Colorable Induced Subgraph}. In order to properly define this problem, we first require some definitions. 

Let $G = (V, E)$ be a finite simple graph. A \textit{coloring} of $G$ is a mapping $\phi \colon V\rightarrow \{1, 2, \ldots\}$ that gives each vertex $u \in V$ a \textit{color} $\phi(u)$ in such a way that, for every two adjacent vertices $u$ and $v$ in $G$, we have that $\phi(u) \neq \phi(v)$. For $r \geq 1$, a coloring $\phi$ of $G$ is an \textit{$r$-coloring} if $\phi(u) \in \{1, \ldots, r\}$ for every $u \in V$, and a graph is \textit{$r$-colorable} if it admits an $r$-coloring. For $r \geq 1$, an \textit{$r$-list assignment} of $G$ is a function $L\colon V \rightarrow 2^{\{1,\ldots,r\}}$ that assigns each vertex $u \in V$ a \textit{list} $L(u) \subseteq \{1,\ldots,r\}$ of admissible colors for $u$. A coloring $\phi$ of $G$ {\it respects} $L$ if  $\phi(u)\in L(u)$ for every $u\in V$. We are finally ready to define \textsc{Max-Weight List $r$-Colorable Induced Subgraph}, where $r$ is a fixed positive integer.

\begin{mybox}{\textsc{Max-Weight List $r$-Colorable Induced Subgraph}}
\textbf{Input:} A graph $G$ equipped with a weight function $w\colon V(G) \rightarrow \mathbb{Q}_{+}$, and an $r$-list assignment $L$ of $G$.\\
\textbf{Task:} Find a subset $F \subseteq V(G)$ such that:
\begin{enumerate}
\item The induced subgraph $G[F]$ admits a coloring that respects $L$, and
\item The weight $w(F) = \sum_{v \in F}w(v)$ is maximum subject to the condition above.
\end{enumerate}
\end{mybox}

\textsc{Max-Weight List $r$-Colorable Induced Subgraph} is a common generalization of several well-known and deeply investigated $\mathsf{NP}$-hard problems, as we explain next. \textsc{Max-Weight List $r$-Colorable Induced Subgraph} generalizes \textsc{List $r$-Coloring} and hence \textsc{$r$-Coloring} as well, which are known to be \textsf{NP}-hard for all $r > 2$ \cite{K72}. Recall that, for a fixed $r \geq 1$, \textsc{List $r$-Coloring} is the problem to decide whether a given graph $G$ with an $r$-list assignment $L$ admits a coloring that respects $L$. By setting $L(u)=\{1,\ldots,r\}$ for every $u\in V(G)$, we obtain {\sc $r$-Coloring}. Note also that, for $r_1 \leq r_2$, \textsc{List $r_1$-Coloring} is a special case of \textsc{List $r_2$-Coloring}.

Several other \textsf{NP}-hard problems are special cases of \textsc{Max-Weight List $r$-Colorable Induced Subgraph} for specific values of $r$. For example, for $r=1$, \textsc{Max-Weight List $r$-Colorable Induced Subgraph} is equivalent to \textsc{Max-Weight Independent Set}, which is the problem of finding a maximum-weight subset of pairwise non-adjacent vertices of an input graph $G$. 
For $r=2$, \textsc{Max-Weight List $r$-Colorable Induced Subgraph} generalizes the problem of finding a maximum-weight induced bipartite subgraph of an input graph $G$ which, by complementation, is equivalent to finding a minimum-weight subset of vertices intersecting all odd cycles in $G$. The latter problem is the well-known \textsc{Odd Cycle Transversal}. 

Given the hardness of \textsc{Max-Weight List $r$-Colorable Induced Subgraph}, it is natural to investigate whether the problem becomes tractable for restricted classes of inputs. The framework of hereditary graph classes (i.e., graph classes closed under vertex deletion) is particularly well suited for this type of research, where the ultimate goal is to obtain complexity dichotomies telling us for which hereditary graph classes the problem at hand can or cannot be solved efficiently (under the standard complexity assumption that $\mathsf{P} \neq \mathsf{NP}$). 

We recall some relevant definitions. A graph $G$ is \textit{$H$-free}, for some graph $H$, if it contains no induced subgraph isomorphic to $H$, that is, we cannot modify $G$ into $H$ by a sequence of vertex deletions. For a set of graphs $\{H_1, \ldots, H_p\}$, a graph is \textit{$(H_1, \ldots, H_p)$-free} if it is $H_i$-free for every $i \in \{1, \ldots, p\}$.
The \textit{disjoint union} $G + H$ of graphs $G$ and $H$ is the graph with vertex set $V(G) \cup V(H)$ and edge set $E(G) \cup E(H)$. We denote the disjoint union of $k$ copies of $G$ by $kG$ and let $P_s$ denote the chordless path on $s$ vertices. 

It is known that, for $r = 2$, \textsc{Max-Weight List $r$-Colorable Induced Subgraph} admits no polynomial-time algorithm on $H$-free graphs unless $H$ is a \textit{linear forest} (i.e., a disjoint union of paths). Indeed, Chiarelli et al.~\cite{chiarelli2018minimum} showed that its special case \textsc{Odd Cycle Transversal} is \textsf{NP}-hard on $H$-free graphs if $H$ contains a cycle or a claw (the claw is the $4$-vertex star). In recent years, considerable work has been done toward classifying the complexity of \textsc{Odd Cycle Transversal} (and its generalizations) on graphs forbidding an induced linear forest.
In \autoref{thm:literature-oct}, we collect known results for \textsc{Odd Cycle Transversal} and its two generalizations  \textsc{Max-Weight $r$-Colorable Induced Subgraph} and \textsc{Max-Weight List $r$-Colorable Induced Subgraph} on $H$-free graphs, where $H$ is a linear forest. The problems are listed in increasing order of generality. In particular, an $\mathsf{NP}$-hardness result for a certain problem implies $\mathsf{NP}$-hardness for a more general problem. Note also that the hardness results hold even in the unweighted setting.

\begin{theorem}\label{thm:literature-oct} The following hold:

\begin{enumerate}[{\rm}(i)]\setlength\itemsep{0.3em}
\item \textsc{Odd Cycle Transversal} on $H$-free graphs can be solved in polynomial time if
	\begin{itemize}
		\item $H = P_5$~\textup{(Agrawal et al.~\cite{agrawal2024odd})}, or 
		\item $H = kP_2$ for all $k\in \mathbb{N}$~\textup{(Chiarelli et al.~\cite{chiarelli2018minimum})}, or 
        \item $H = P_3 + kP_1$ for all $k\in \mathbb{N}$~\textup{(Dabrowski et al.~\cite{dabrowski2020cycle})},
	\end{itemize}
	and remains \textsf{NP}-hard if
	\begin{itemize}
		\item $H = (P_6, P_5+P_2)$~\textup{(Dabrowski et al.~\cite{dabrowski2020cycle})}.
	\end{itemize}
\item \textsc{Max-Weight $r$-Colorable Induced Subgraph} on $H$-free graphs can be solved in polynomial time if
\begin{itemize}
		\item $H = P_5 + kP_1$ for all $k\in \mathbb{N}$~\textup{(Henderson et al.~\cite{henderson2024maximum})},
\end{itemize}
	and remains \textsf{NP}-hard if
	\begin{itemize}
		\item $H = (P_6, P_5+P_2)$ for $r=2$, or
		\item $H = 2P_4$ for all $r\geq5$ \textup{(Hajebi et al.~\cite{hajebi2022complexity})}.
	\end{itemize}
\item \textsc{Max-Weight List $r$-Colorable Induced Subgraph} on $H$-free graphs can be solved in polynomial time if
\begin{itemize}
		\item $H = P_5$ \textup{(Lokshtanov et al.~\cite{lokshtanov2024maximum})},
            \item $H = P_5 + kP_1$ for all $k\in \mathbb{N}$~\textup{(Henderson et al.~\cite{henderson2024maximum})},
\end{itemize}
	and remains \textsf{NP}-hard if
	\begin{itemize}
		\item $H = (P_6, P_5+P_2)$ for $r\geq2$, or
		\item $H = P_4+P_2$ for all $r\geq 5$ \textup{(Couturier et al.~\cite{couturier2015list})}.
	\end{itemize}
\end{enumerate}
\end{theorem}

Recently, Chudnovsky et al.~\cite{CHS24} obtained the following complete complexity dichotomy for \textsc{List $r$-Coloring} when $r \geq 5$. Assuming $\mathsf{P}\neq\mathsf{NP}$, \textsc{List $r$-Coloring} ($r \geq 5$) can be solved in polynomial time on $H$-free graphs if and only if $H$ is an induced subgraph of either $kP_3$ or $P_5 + kP_1$, for some $k \in \mathbb{N}$. Their main result toward this was showing that \textsc{List $r$-Coloring} ($r \geq 1$) can be solved in polynomial time on $kP_3$-free graphs, for any $k \in \mathbb{N}$, and this was obtained building on a very technical result of Hajebi et al.~\cite[Theorem~5.1]{hajebi2022complexity}.   

Motivated by the quest for a complexity dichotomy, Agrawal et al.~\cite{agrawal2024odd} posed very recently as an open problem to classify the computational complexity of \textsc{Odd Cycle Transversal} on $(P_3+P_2)$-free graphs, the unique minimal open case stemming from \Cref{thm:literature-oct}.

It should also be mentioned that classifying the complexity of \textsc{Max-Weight Independent Set} on $H$-free graphs when $H$ is a linear forest is a notorious open problem in algorithmic graph theory (see \cite{CMPPR24} for the state of the art). Note however that \textsc{Max-Weight Independent Set} substantially differs from \textsc{Odd Cycle Transversal} on $H$-free graphs, in the sense that it is polynomial-time solvable on claw-free graphs.

In this paper, we also consider the distance-$d$ generalizations of \textsc{Max-Weight Independent Set} and \textsc{List $r$-Coloring}, defined as follows. For $d \geq 2$, a \textit{distance-$d$ independent set} of a graph $G$ is a set of vertices of $G$ pairwise at distance at least $d$ in $G$. For fixed $d \geq 2$, \textsc{Max-Weight Distance-$d$ Independent Set} (also known as $d$-\textsc{Scattered Set}) is the problem to find a maximum-weight distance-$d$ independent set of an input graph $G$.

Trivially, for every fixed $d \geq 2$, \textsc{Max-Weight Distance-$d$ Independent Set} is easy on $P_{d+1}$-free graphs, and the following hardness results are known.

\begin{theorem}\label{eto} 
The following hold for \textsc{Distance-$d$ Independent Set} on $H$-free graphs:
\begin{enumerate}[{\rm}(i)]\setlength\itemsep{0.3em}
\item The problem is $\mathsf{NP}$-hard if $H$ contains $2P_{(d+1)/2}$, for every fixed odd $d \geq 3$ \textup{(Eto et al.~\cite[Theorem~4]{EGM14})}; 
\item Assuming ETH, the problem admits no $2^{o(|V(H)|+|E(H)|)}$-time algorithm if $H$ contains a cycle or a claw, for every fixed $d \geq 3$ \textup{(Bacs{\'{o}} et al.~\cite{BLMPTL19})}.
\end{enumerate}   
\end{theorem}

For $d \geq 2$, a \textit{$(d,r)$-coloring} of a graph $G$ is an assignment of colors to the vertices of $G$ using at most $r$ colors such that no two distinct vertices at distance less than $d$ receive the same color. Thus, a $(2,r)$-coloring is nothing but an $r$-coloring. Similarly as above, for fixed $d \geq 2$ and $r \geq 1$, we define  \textsc{$(d,r)$-Coloring} as the problem of determining whether a given graph $G$ has a $(d, r)$-coloring. \textsc{List $(d,r)$-Coloring} is defined similarly but we require in addition that every vertex $u$ must receive a color from some given list $L(u) \subseteq \{1,\ldots, r\}$. 

Sharp~\cite{Sharp07} provided the following complexity dichotomy: For fixed $d \geq 3$, \textsc{$(d,r)$-Coloring} is polynomial-time solvable for $r \leq \lfloor 3d/2 \rfloor$ and $\mathsf{NP}$-hard for $r > \lfloor 3d/2 \rfloor$.

To the best of our knowledge, only one complete complexity dichotomy ($\mathsf{P}$ vs $\mathsf{NP}$-hard) on $H$-free graphs is known for a ``problem at distance'', obtained by Dallard et al.~\cite{DKM21} for \textsc{Distance-$d$ Vertex Cover}.

\subsection*{Our results} We prove three main algorithmic results for $kP_3$-free graphs. The first result, proven in \Cref{sec:amiable-fam}, concerns \textsc{Max-Weight List $r$-Colorable Induced Subgraph}.

\begin{restatable}{theorem}{main}\label{thm:main}
Let $r \geq 1$ be a fixed integer. For every $k \in \mathbb{N}$, \textsc{Max-Weight List $r$-Colorable Induced Subgraph} can be solved in polynomial time on $kP_3$-free graphs.
\end{restatable}

\Cref{thm:main} has several interesting consequences. First, it immediately implies that \textsc{Odd Cycle Transversal} can be solved in polynomial time on $kP_3$-free graphs, for every $k \in \mathbb{N}$, thus solving a generalized version of the aforementioned open problem of Agrawal et al.~\cite{agrawal2024odd}. Our result for \textsc{Odd Cycle Transversal} also complements the polynomial-time algorithms for \textsc{Feedback Vertex Set} and \textsc{Even Cycle Transversal} on $kP_3$-free graphs of Paesani et al.~\cite{PPR22}. Second, \Cref{thm:main} generalizes the recent result of Chudnovsky et al.~\cite{CHS24} that \textsc{List $r$-Coloring} can be solved in polynomial time on $kP_3$-free graphs for every $r,k \in \mathbb{N}$. Although partially inspired by their approach, as we explain below, our proof of the more general \Cref{thm:main} has the advantage of being considerably shorter and self-contained.  

\Cref{thm:main} also makes considerable progress toward a complete complexity dichotomy for \textsc{Max-Weight List $r$-Colorable Induced Subgraph} and \textsc{Odd Cycle Transversal} on $H$-free graphs. Indeed, paired with the recent result of \cite{henderson2024maximum}, it \textit{completely} settles the complexity of \textsc{Max-Weight List $r$-Colorable Induced Subgraph} on $H$-free graphs for $r\geq 5$ (see \Cref{thm:literature-oct} and the discussion preceding it): 

\begin{theorem} Let $r \geq 5$ be a fixed integer. Assuming $\mathsf{P}\neq\mathsf{NP}$, \textsc{Max-Weight List $r$-Colorable Induced Subgraph} on $H$-free graphs is polynomial-time solvable if and only if $H$ is an induced subgraph of either $kP_3$ or $P_5+kP_1$, for some $k \geq 1$.
\end{theorem}

Moreover, paired with the results of \cite{chiarelli2018minimum,dabrowski2020cycle,henderson2024maximum} mentioned above, \Cref{thm:main} leaves the case $H=k_4P_4+k_3P_3+k_2P_2+k_1P_1$, where $k_4 \geq 1$ and $k_4 + k_3 \geq 2$, as the \textit{only} remaining open case toward a complete complexity dichotomy for \textsc{Odd Cycle Transversal} on $H$-free graphs. 

We then consider, in \Cref{sec:distancedamiable}, the distance-$d$ generalizations of \textsc{Max-Weight Independent Set} and \textsc{List $r$-Coloring}, for $d \geq 6$, and prove the following two results.

\begin{restatable}{theorem}{distanceIS}\label{thm:distanceIS}
Let $d \geq 6$ be a fixed integer. For every $k \in \mathbb{N}$, \textsc{Max-Weight Distance-$d$ Independent Set} can be solved in polynomial time on $kP_3$-free graphs.
\end{restatable}

\begin{restatable}{theorem}{distanceCol}\label{thm:distanceCol}
Let $d \geq 6$ and $r \geq 1$ be fixed integers. For every $k \in \mathbb{N}$, \textsc{List $(d,r)$-Coloring} can be solved in polynomial time on $kP_3$-free graphs.
\end{restatable}

Paired with the aforementioned result of Eto et al.~\cite{EGM14}, \Cref{thm:distanceIS} completely settles the computational complexity of \textsc{Max-Weight Distance-$d$ Independent Set} on $kP_3$-free graphs, except for the \textit{only} remaining open case $d=4$.

\subsubsection*{Technical overview}

We now explain our approach toward \Cref{thm:main}, which combines ideas from \cite{L17,LM12} and \cite{CHS24}. It is instructive to first consider the special case of \textsc{Odd Cycle Transversal} (by complementation, \textsc{Max-Weight $2$-Colorable Induced Subgraph}) on $kP_2$-free graphs, where a very simple algorithm can be obtained. Indeed, $kP_2$-free graphs have polynomially many (inclusion-wise) maximal independent sets, and these can be enumerated in polynomial time. Consequently, a maximum-weight induced bipartite subgraph can be found by exhaustively enumerating all pairs of maximal independent sets \cite{chiarelli2018minimum}. However, it is easily seen that even $P_3$-free graphs (i.e., graphs whose connected components are cliques) do not have polynomially many maximal independent sets. But it turns out that a weaker property is enough for our purposes: Admitting a polynomial family of ``well-behaved'' vertex sets such that every maximal independent set is contained in one of these sets. In our case, ``well-behaved'' means inducing a $P_3$-free subgraph. The intuition is that, given such a family, we can efficiently guess the color classes, each of which will be a disjoint union of cliques, and then match vertices to the possible color classes.


The following key notion, which we dub amiable family, was first introduced by Lozin and Mosca~\cite{LM12}. For a graph $G$, a family $\mathcal{S} \subseteq 2^{V(G)}$ of subsets of $V(G)$ is an \textit{amiable family} if it satisfies the
following two properties:
\begin{itemize}
	\item Each member of $\mathcal{S}$ induces a $P_3$-free subgraph in $G$;
\item Each (inclusion-wise) maximal independent set of $G$ is contained in some member of $\mathcal{S}$.
\end{itemize}

Lozin and Mosca~\cite{LM12} showed that, when $k=2$, every $kP_3$-free graph $G$ admits an amiable family of size polynomial in $|V(G)|$ and which can be computed in polynomial time. Later, Lozin~\cite{L17} observed how such property in fact holds for every $k \geq 2$ (see \Cref{lem:amiablefam} for a formal statement). Given an amiable family $\mathcal{S}$ of polynomial size of a $kP_3$-free graph $G$, we would like to exhaustively solve \textsc{Max-Weight List $r$-Colorable Induced Subgraph} on every possible $r$-tuple consisting of members of $\mathcal{S}$. More precisely, let $(S_1, \ldots, S_r) \in \mathcal{S}^r$ be an $r$-tuple of members of $\mathcal{S}$. We would like to find a maximum-weight induced subgraph of $G[\bigcup_{i \in [r]} S_i]$ which admits an $r$-coloring respecting the given $r$-list assignment and such that, for $i=1, \ldots, r$, all vertices colored $i$ are contained in $S_i$. To do this, we then extend an idea of Chudnovsky et al.~\cite{CHS24} as follows. We reduce our problem to that of finding a maximum-weight matching in an auxiliary bipartite graph where one partition class $Y$ consists of $\bigcup_{i \in [r]} S_i$, the other class $X$ consists of the connected components of the subgraphs $G[S_i]$, for $i = 1, \ldots, r$, and there is an edge between $y \in Y$ and $x \in X$ if and only if $y$ belongs to the connected component $x$. Since each weighted matching problem can be solved in polynomial time using the Hungarian method (see, e.g., \cite[Theorem~17.3]{Sch}) and we build $|\mathcal{S}|^r$ auxiliary problems, which is a polynomial in $|V(G)|$, a solution to \textsc{Max-Weight List $r$-Colorable Induced Subgraph} can be found in polynomial time. 

\begin{figure}
\captionsetup{width=\linewidth}
\hspace{-2cm}
    \includegraphics[width=0.8\linewidth]{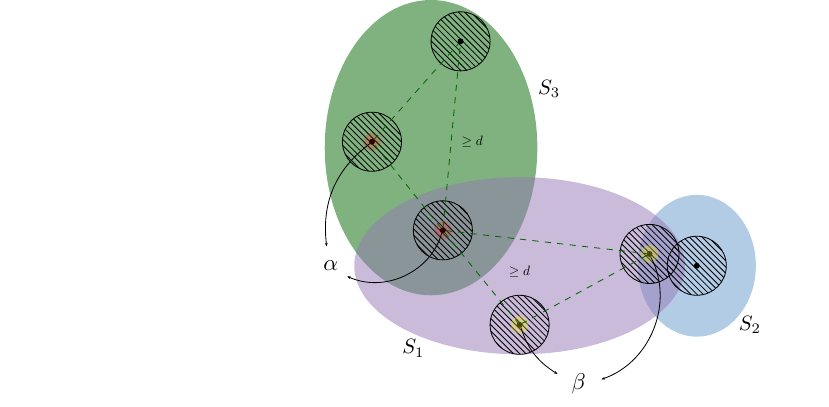}
    \caption{Visualization for distance-$d$ amiable family $\mathcal{S} = \{S_1, S_2, S_3\}$. Circles represent cliques and $\alpha$, $\beta$ are maximal distance-$d$ independent sets. Dashed lines depict paths of lengths at least $d$.}
    \label{fig:distance-d-amiable}
\end{figure}

In order to prove \Cref{thm:distanceIS,thm:distanceCol}, we consider the following distance-$d$ generalization of the notion of amiable family. For a graph $G$, a family $\mathcal{S} \subseteq 2^{V(G)}$ of subsets of $V(G)$ is a \textit{distance-$d$ amiable family} if it satisfies the following properties:
\begin{itemize}
    \item Each member of $\mathcal{S}$ induces a $P_3$-free subgraph in $G$;
    \item For each $S \in \mathcal{S}$, the connected components of $G[S]$ are pairwise at distance at least $d$ in $G$;
    \item Each (inclusion-wise) maximal distance-$d$ independent set of $G$ is contained in some member of $\mathcal{S}$.
\end{itemize}

Clearly, a distance-$2$ amiable family is nothing but an amiable family. Our main technical contribution is the following. 

\begin{restatable}{lemma}{distamiablefam}\label{lem:dist-amiablefam}
Let $d\geq 6$ be a fixed integer. For every $k \in \mathbb{N}$, every $kP_3$-free graph admits a distance-$d$ amiable family of size $|V(G)|^{O(k)}$, which can be computed in time $|V(G)|^{O(k)}$.
\end{restatable}

Although the algorithm for \Cref{lem:dist-amiablefam} is inspired by the case $d = 2$ and hence by the work of Lozin and Mosca~\cite{LM12}, its proof of correctness is much more involved and requires genuinely new ideas.

Equipped with \Cref{lem:dist-amiablefam}, we immediately obtain \Cref{thm:distanceIS}. Indeed, in order to solve \textsc{Max-Weight Distance-$d$ Independent Set} on a $kP_3$-free graph $G$, we simply find in polynomial time a distance-$d$ amiable family $\mathcal{S}$ of $G$ as above and, for each member $S \in \mathcal{S}$, find a max-weight independent set in the $P_3$-free graph $G[S]$. The latter can be clearly done in polynomial time, thus proving \Cref{thm:distanceIS}. The proof of \Cref{thm:distanceCol} is similar to that of \Cref{thm:main}, the only difference being the use of a distance-$d$ amiable family for $d \geq 6$.   

A remark about our approach, which combines ideas of Lozin and Mosca~\cite{L17,LM12} and Chudnovsky et al.~\cite{CHS24}, is in place. The proof of Chudnovsky et al.~\cite{CHS24} that \textsc{List $r$-Coloring} ($r \geq 1$) is polynomial-time solvable on $kP_3$-free graphs generalizes the earlier result of Hajebi et al.~\cite{hajebi2022complexity} that \textsc{List $5$-Coloring} is polynomial-time solvable on $kP_3$-free graphs, by replacing the second step in the proof of Hajebi et al.~(see \cite[Theorem~4.3]{hajebi2022complexity}) with a significantly simpler argument (the aforementioned reduction to a bipartite matching problem) that, in addition, works for all $r \in \mathbb{N}$. However, their result still relies on the very technical first step of Hajebi et al.~(see \cite[Theorem~5.1]{hajebi2022complexity}). Our approach can be viewed as a step further in the direction of simplifying and generalizing, as exemplified by our \Cref{thm:main,thm:distanceCol}, which extend in different ways the main result of Chudnovsky et al.~\cite{CHS24} by means of an arguably elegant and self-contained proof.

We conclude with some remarks on possible alternative approaches. As already mentioned, Paesani et al.~\cite{PPR22} provided polynomial-time algorithms for \textsc{Feedback Vertex Set} (by complementation, the problem of finding a max-weight induced forest) and \textsc{Even Cycle Transversal} (by complementation, the problem of finding a max-weight induced odd cactus) on $kP_3$-free graphs. As they observe, their technique cannot be adapted for \textsc{Odd Cycle Transversal}. Loosely speaking, the reason is that the blocks of bipartite graphs can be arbitrarily complicated, contrary to the case of forests and odd cacti.  

Some of the problems mentioned in this introduction are polynomial-time solvable on graph classes with polynomially many potential maximal cliques (or, equivalently \cite{BT02}, polynomially many minimal separators) \cite{FTV15,MT16}, on graph classes of bounded tree-independence number \cite{LMM24} and, more generally, of bounded sim-width \cite{BMPY25}. However, even $2P_3$-free graphs do not have polynomially many minimal separators (see, e.g., \cite{MP21}). Similarly, $2P_3$-free graphs have unbounded sim-width (see, e.g., \cite{BKR23}) and hence unbounded tree-independence number as well.


\section{Preliminaries}\label{sec:prelim}
We denote the set of positive integers by $\mathbb{N}$. For every $n\in \mathbb{N}$, we let $[n]:= \{1,\dots, n\}$. Given a set $A$, we denote by $A^r$ the set of all ordered $r$-tuples of elements of $A$, i.e., $A^r = \{(a_1, \ldots, a_r) : a_i \in A \ \mbox{for} \ i = 1, \ldots, r\}$. 

All graphs in our paper are finite and simple. The \textit{empty graph} is the graph with no vertices.
Let now $G$ be a graph. For $X \subseteq V(G)$, we denote the subgraph of $G$ \textit{induced} by $X$ as $G[X]$, that is $G[X] = (X, \{uv : u,v \in X \ \mbox{and} \ uv \in E(G)\})$. We use $N_{G}(X)$ to denote the neighbors in $V\setminus X$ of vertices in $X$. For disjoint sets $X, Y \subseteq V(G)$, we say that $X$ is \textit{complete} to $Y$ if every vertex in $X$ is adjacent to every vertex in $Y$, and $X$ is \textit{anticomplete} to $Y$ if there are no edges between $X$ and $Y$. For a subset $X\subseteq V(G)$, the \textit{anti-neighborhood} of $X$, denoted by $\mathcal{A}(X)$, is the subset of vertices in $V(G)\setminus X$ which are anticomplete to $X$. With a slight abuse of notation, if $X = \{v_1, \ldots, v_i\}$, we denote $\mathcal{A}(X) = \mathcal{A}(\{v_1, \ldots, v_i\})$ by $\mathcal{A}(v_1, \ldots, v_i)$. 

Given two subsets $X, Y \subseteq V(G)$, an \textit{$X, Y$-path} in $G$ is a path in $G$ which has one end in $X$, the other end in $Y$, and whose
inner vertices belong to neither $X$ nor $Y$. For vertices $u, v \in V(G)$, we denote by $d_{G}(u, v)$ the distance between $u$ and $v$ in $G$, i.e., the length of a shortest $u, v$-path in $G$. If no such path exists, we let $d_G(u,v) = \infty$. Moreover, for $d \in \mathbb{N}$, we let $N^{\geq d}_G(v) = \{u\in V(G) : d_G(v,u) \geq d\}$ and $N^{\leq d}_G(v) = \{u\in V(G)\setminus\{v\} :  d_G(v,u) \leq d\}$. 
Given two subsets $X, Y \subseteq V(G)$, the \textit{distance} between $X$ and $Y$ in $G$ is the quantity $d_G(X,Y) = \min_{x\in X,y\in Y} d_G(x, y)$, i.e., the length of a shortest path in $G$ between a vertex in $X$ and a vertex in $Y$. Given $u \in V(G)$ and $Q \subseteq V(G)$, we say that $u$ is connected to $Q$ in $G$ if there exists a $u, Q$-path in $G$ (possibly of length $0$).

A \textit{clique} of a graph is a set of pairwise adjacent vertices and an \textit{independent set} is a set of pairwise non-adjacent vertices. A \textit{matching} of a graph is a set of pairwise
non-adjacent edges.

A \textit{connected component} of $G$ is a maximal connected subgraph of $G$. For convenience, we will often view a connected component as its vertex set rather than the subgraph itself. For this reason, we will say for example that ``a connected component is a clique'' rather than ``a connected component is a complete subgraph''. Throughout the paper, we will repeatedly make use of the fact that every connected component of a $P_3$-free graph is a clique.  


\section{The proof of \Cref{thm:main}}\label{sec:amiable-fam}

In this section we prove \Cref{thm:main}. As mentioned in the introduction, the first step is to obtain a polynomial-time algorithm for finding an amiable family (necessarily of polynomial size) of a $kP_3$-free graph. Since no proof for $k > 2$ is given in \cite{L17} and as a preparation for the more technical case of distance-$d$ amiable families addressed in \Cref{sec:distancedamiable}, we provide a full proof of this result in \Cref{lem:amiablefam}. The second step consists then in reducing \textsc{Max-Weight List $r$-Colorable Induced Subgraph} to polynomially many auxiliary weighted matching problems. We do this in \Cref{lem:good-graphs}. We finally combine these two steps and prove \Cref{thm:main}.

\begin{lemma}\label{lem:amiablefam}
For every $k \in \mathbb{N}$, every $kP_3$-free graph $G$ admits an amiable family of size $|V(G)|^{O(k)}$, which can be computed in time $|V(G)|^{O(k)}$. 
\end{lemma}

\begin{proof}
To prove the lemma, we provide, for every $k \in \mathbb{N}$, an algorithm $\Gamma_{k}$ which takes as input a $kP_3$-free graph $G$ together with an arbitrary ordering of $V(G)$ and outputs an amiable family $\Gamma_k(G)$ of $G$. The pseudo-code of the algorithm $\Gamma_k$ is given in \Cref{algo:Gamma_k}. Note that, if the input graph $G$ is the empty graph, $\Gamma_k$ correctly outputs $\Gamma_k(G) = \{\varnothing\}$. Moreover, for $k=1$, $\Gamma_k$ correctly outputs $\{V(G)\}$. We now show that, for every $k \in \mathbb{N}$ and every $kP_3$-free graph $G$, the family returned by $\Gamma_k$ is indeed an amiable family of $G$. 

In the following, given a graph $G$ on $n$ vertices and an ordering $v_1,\ldots,v_n$ of $V(G)$, we let $G_i = G[\{v_1,\ldots,v_i\}]$. Moreover, we denote an induced $P_3$ in $G$ with vertex set $\{u,v,w\}$ and degree-$2$ vertex $v$ by $uvw$.  

\begin{algorithm}{}
\caption{$\Gamma_k$}
\label{algo:Gamma_k}
\begin{algorithmic}[1]
\Require A $kP_3$-free graph $G$ with an arbitrary ordering $v_1, \ldots, v_n$ of $V(G)$. 
\Ensure An amiable family of $G$.
\State Initialize $\mathcal{S} = \{\varnothing\}$.
\For{$i = 1, \dots, n$}\label{mainloop}

\For{every member $S \in \mathcal{S}$}\label{step2.1} \Comment{Extend members of $\mathcal{S}$}
\If{$G[S\cup \{v_i\}]$ is $P_3$-free}
\State Set $S = S \cup \{v_i\}$.
\EndIf
\EndFor

\For{every induced $P_3$ $v_{i}uw$ in $G_i$}\label{step2.2} \Comment{Add new members to $\mathcal{S}$}
\State Compute the family $\mathcal{C}:=\Gamma_{k-1}(G[\mathcal{A}(v_{i},u,w)])$. 
\For{every $C \in \mathcal{C}$} 
\State Set $\mathcal{S} = \mathcal{S} \cup \{C \cup \{v_i,w\}\}$.
\EndFor
\EndFor

\For{every induced $P_3$ $uv_{i}w$ in $G_i$}\label{step2.3} \Comment{Add new members to $\mathcal{S}$}
\State Compute the family $\mathcal{C}:=\Gamma_{k-1}(G[\mathcal{A}(v_{i},u,w)])$. 
\For{every $C \in \mathcal{C}$} 
\State Set $\mathcal{S} = \mathcal{S} \cup \{C \cup \{v_i\}\}$.
\EndFor
\EndFor
\EndFor
\State \Return $\Gamma_k(G) := \mathcal{S}$.
\end{algorithmic}
\end{algorithm}

\clm{For every $k \in \mathbb{N}$, if $G$ is a $kP_3$-free graph and $S\in \Gamma_k(G)$, then $G[S]$ is $P_3$-free.}\label{amiablefam:P3-freeness}

\begin{claimproof}[Proof of \eqref{amiablefam:P3-freeness}] We proceed by induction on $k$. The statement holds for $k=1$, since the graph in input itself is $P_3$-free. Suppose now that $k > 1$ and that the statement holds for $k-1$. Let $G$ be an $n$-vertex $kP_3$-free graph in input. For every $i \in [n]$, let $\mathcal{S}^i$ be the state of the family $\mathcal{S}$ at the end of the $i$-th iteration of the main loop in line~\ref{mainloop}. It is enough to show by induction on $i$ that each member of $\mathcal{S}^i$ induces a $P_3$-free subgraph of $G$. The base case $i=1$ trivially holds, since $\mathcal{S}^1 = \{\{v_1\}\}$. Thus, suppose that $i > 1$ and that the statement holds for $i-1$.

Consider $S \in \mathcal{S}^i$. If $S \in \mathcal{S}^{i-1}$ then, by the induction hypothesis, $G[S]$ is $P_3$-free. Thus, we may assume that $S \notin \mathcal{S}^{i-1}$. This implies that $S$ is added to $\mathcal{S}^i$ in one of the three inner loops during the $i$-th iteration of the main loop. Suppose first that $S$ is added in the line-\ref{step2.1} loop as an extension of a member of $\mathcal{S}^{i-1}$. Then, $G[S]$ is $P_3$-free by construction. Suppose next that $S$ is added to $\mathcal{S}^i$ in the line-\ref{step2.2} loop. This implies that there exists an induced $P_3$ $v_iuw$ in $G_i$ and a set $C \in \Gamma_{k-1}(G[\mathcal{A}(v_i,u,w)])$ such that $S = C \cup \{v_i,w\}$. Observe that, since $G$ is $kP_3$-free, $G[\mathcal{A}(v_i,u,w)]$ is $(k-1)P_3$-free and so, by the induction hypothesis on $k-1$, every member of $\Gamma_{k-1}(G[\mathcal{A}(v_i,u,w)])$ induces a $P_3$-free subgraph of $G$. This readily implies that $G[C \cup \{v_i,w\}]$ is $P_3$-free as well. We conclude similarly in the case that $S$ is added to $\mathcal{S}^i$ in the line-\ref{step2.3} loop.
\end{claimproof}

\clm{For every $k \in \mathbb{N}$, if $G$ is a $kP_3$-free graph, then every independent set of $G$ is contained in a member of $\Gamma_k(G)$.}\label{amiablefam:IS}

\begin{claimproof}[Proof of \eqref{amiablefam:IS}] We proceed by induction on $k$. The statement clearly holds for $k=1$, since for any $P_3$-free graph $G$, $\Gamma_1(G) = \{V(G)\}$. Suppose now that $k > 1$ and that the statement holds for $k-1$. Let $G$ be a $kP_3$-free graph in input. For every $i \in [n]$, let $\mathcal{S}^i$ be the state of the family $\mathcal{S}$ at the end of the $i$-th iteration of the main loop. It is enough to show by induction on $i$ that every independent set of $G_i$ is contained in some member of $\mathcal{S}^i$. The base case $i = 1$ trivially holds, since $\mathcal{S}^1 = \{\{v_1\}\}$ and $\{v_1\}$ is the only maximal independent set of $G_1$. Thus, suppose that $i > 1$ and that the statement holds for $i - 1$.

Consider an independent set $I$ of $G_i$. If $v_i \notin I$, then $I$ is an independent set of $G_{i-1}$ and so, by the induction hypothesis, there exists $S \in \mathcal{S}^{i-1}$ such that $I \subseteq S$. Observe now that, by construction, there exists $S' \in \mathcal{S}^i$ such that $S \subseteq S'$ and hence $I \subseteq S'$. If however $v_i \in I$ then, by the induction hypothesis, there exists $S \in \mathcal{S}^{i-1}$ such that $I \setminus \{v_i\} \subseteq S$. If $G[S \cup \{v_i\}]$ is $P_3$-free, then $S \cup \{v_i\} \in \mathcal{S}^i$ thanks to the line-\ref{step2.1} loop. Thus, we may assume that $G[S \cup \{v_i\}]$ contains at least one induced $P_3$. Moreover, by the proof of \eqref{amiablefam:P3-freeness}, $G[S]$ is $P_3$-free, and so any such $P_3$ must contain $v_i$. We distinguish two cases.

Suppose first that there exists a connected component $Q$ of $G[S]$ and a vertex $u \in Q$ such that $v_i$ is adjacent to $u$ but not complete to $Q$. Note that, since $G[S]$ is $P_3$-free, $Q$ is a clique. Now, if $Q$ contains a vertex of $I$, let $w$ be this vertex; otherwise, let $w$ be an arbitrary vertex of $Q$ nonadjacent to $v_i$. Observe that $I \setminus \{v_i,w\}$ is fully contained in $G[S] - Q$, which implies that $I \setminus \{v_i,w\} \subseteq \mathcal{A}(v_i,u,w)$. Since $G[\mathcal{A}(v_i,u,w)]$ is $(k-1)P_3$-free, the induction hypothesis on $k-1$ implies that there exists $C \in \Gamma_{k-1}(G[\mathcal{A}(v_i,u,w)])$ such that $I \setminus \{v_i,w\} \subseteq C$. Hence, $C \cup \{v_i,w\} \in \mathcal{S}^i$ and $I \subseteq C \cup \{v_i,w\}$. 

Suppose finally that, for every connected component $Q$ of $G[S]$, the vertex $v_i$ is either complete or anticomplete to $Q$. This implies that there exist at least two connected components $Q_1$ and $Q_2$ of $G[S]$ such that $v_i$ is complete to both $Q_1$ and $Q_2$. Let $u$ and $w$ be two arbitrary vertices of $Q_1$ and $Q_2$, respectively. Since $I \setminus \{v_i\}$ is fully contained in $G[S] - (Q_1 \cup Q_2)$, in particular, $I \setminus \{v_i\} \subseteq \mathcal{A}(v_i,u,w)$. But $G[\mathcal{A}(v_i,u,w)]$ is $(k-1)P_3$-free and so, by the induction hypothesis on $k-1$, there exists $C \in \Gamma_{k-1}(G[\mathcal{A}(v_i,u,w)])$ such that $I \setminus \{v_i\} \subseteq C$. Hence, $C \cup \{v_i\} \in \mathcal{S}^i$ and $I \subseteq C \cup \{v_i\}$, thus concluding the proof of \eqref{amiablefam:IS}. 
\end{claimproof}

It now follows from \eqref{amiablefam:P3-freeness} and \eqref{amiablefam:IS} that, for every $k \in \mathbb{N}$, if $G$ is a $kP_3$-free graph, then the family $\Gamma_k(G)$ is indeed an amiable family. It remains to show that $\Gamma_k(G)$ has size at most $|V(G)|^{O(k)}$ and that the running time of the algorithm $\Gamma_k$ is $|V(G)|^{O(k)}$. To this end, for every $n,k \in \mathbb{N}$, let $f(n,k) = \max \{|\Gamma_k(G)| \colon |V(G)| \leq n \text{ and } G \text{ is } kP_3\text{-free}\}$. Clearly, $f(n,1) = 1$ for every $n \in \mathbb{N}$. We claim that, for every $n \in \mathbb{N}$ and $k > 1$, $f(n,k) \leq 2n^4\cdot f(n,k-1)$. Indeed, for every $n$-vertex $kP_3$-free graph $G$, a member of $\Gamma_k(G)$ can only be created during the $i$-th iteration of the main loop (for some $i \in [n]$), in one of the inner loops of lines \ref{step2.2} and \ref{step2.3} from an induced $P_3$ of $G_i$ and some set resulting from a call to $\Gamma_{k-1}$. Since for each $i \in [n]$ there are at most $i^3$ such copies of $P_3$ in $G_i$, at most $2i^3 \cdot f(n,k-1)$ new members are added in the $i$-th iteration of the main loop. It follows that $f(n,k) \leq \sum_{i \in [n]} 2i^3 \cdot f(n,k-1) \leq 2n^4 \cdot f(n,k-1)$ and thus, $f(n,k) \leq 2^{k-1}n^{4(k-1)}$. Similarly, for every $n,k \in \mathbb{N}$, if $T(n,k)$ denotes the running time of the algorithm $\Gamma_k$ on an $n$-vertex $kP_3$-free graph, then clearly $T(n,1) = O(n)$. Furthermore, we obtain the following recurrence for $T(n,k)$, where we use the fact that checking if an $n$-vertex graph is $P_3$-free can be done in $O(n^3)$ time: \[T(n,k) \leq cn\cdot(f(n,k)\cdot n^3 + 2n^3\cdot(T(n,k-1) + f(n,k-1))) \leq 2cn^4 \cdot T(n,k-1) + O(n^{4k}),\] for some constant $c > 0$. We conclude that $T(n,k) \leq n^{O(k)}$. This completes the proof of~\Cref{lem:amiablefam}.
\end{proof}

The next result constitutes the second step in the proof of \Cref{thm:main}. We formulate it for arbitrary distances, as it will also be used in the proof of \Cref{thm:distanceCol} given in the next section.

\begin{lemma}\label{lem:good-graphs}
Let $r\geq 1$ and $d \geq 2$ be fixed integers. Let $G$ be a graph with weight function $w\colon V(G) \rightarrow \mathbb{Q}_{+}$, $L$ an $r$-list assignment of $G$, and $\mathcal{S}$ a distance-$d$ amiable family of $G$. Given an $r$-tuple $(S_1,\ldots,S_r) \in \mathcal{S}^r$, there exists an $O(((r+1)|V(G)|)^{3})$-time 
algorithm that finds a maximum-weight induced subgraph $H$ of $G[\bigcup_{i \in [r]} S_i]$ which admits a $(d, r)$-coloring $\phi\colon V(H) \to [r]$ satisfying the following:
\begin{enumerate}[label=(\roman*),ref=(\roman*)]
\item\label{itm:property1} For every $v \in V(H)$, $\phi(v) \in L(v)$;
\item\label{itm:property2} For every $i \in [r]$, $\{v \in V(H) \colon \phi(v) = i\} \subseteq S_i$.
\end{enumerate}
\end{lemma}

\begin{proof}
Consider an $r$-tuple $(S_1,\ldots,S_r) \in \mathcal{S}^r$. For every $i \in [r]$, let $c_i$ be the number of connected components of $G[S_i]$ and let $S_i^1,\ldots,S_i^{c_i}$ be an arbitrary ordering of the connected components of $G[S_i]$. By definition of distance-$d$ amiable family, each such connected component of $G[S_i]$ is a clique and any two of them are pairwise at distance at least $d$ in $G$.

The first step of our algorithm consists in preprocessing the graph $G[\bigcup_{i \in [r]} S_i]$ as follows. For every $i \in [r]$, if there exists a vertex $v \in S_i$ such that $i \notin L(v)$, then we remove $v$ from $S_i$, that is, we set $S_i = S_i \setminus \{v\}$. Observe that this preprocessing is safe. Indeed, if $H$ is an induced subgraph of $G[\bigcup_{i \in [r]} S_i]$ for which there exists a $(d,r)$-coloring $\phi$ satisfying both \ref{itm:property1} and \ref{itm:property2}, then surely $\phi(v) \neq i$ if $v \in V(H)$.

We next show that finding such an induced subgraph of $G[\bigcup_{i \in [r]} S_i]$ of maximum weight boils down to finding a maximum-weight matching in an auxiliary weighted bipartite graph $B$ constructed as follows. The graph $B$ has bipartition $X \cup Y$ and edge set $E(B)$, where 
\[X = \Big\{x_{S_i^j} : i\in [r], j\in [c_i] \Big\}, \ Y = \Big\{y_{v} : v\in \bigcup_{i \in [r]} S_i\Big\} \ \mbox{and} \  
E(B) = \bigcup_{i=1}^{r} \bigcup_{j=1}^{c_i} \Big\{y_{v}x_{S_i^j} \colon v\in S_i^{j} \Big\}.\]
Moreover, for each $v\in \bigcup_{i \in [r]} S_i$, every edge of $B$ incident to $y_v$ is assigned the weight $w(v)$.
Note that $|V(B)| = |\bigcup_{i \in [r]} S_i| + \sum_{i \in [r]} c_i \leq |V(G)| + r|V(G)| = (r+1)|V(G)|$.

\clm{Let $m \in \mathbb{Q}_+.$ The graph $G[\bigcup_{i \in [r]} S_i]$ has an induced subgraph $H$ with $w(V(H)) = m$ and which admits a $(d,r)$-coloring satisfying both \ref{itm:property1} and \ref{itm:property2} if and only if $B$ has a matching of weight $m$.}\label{clm:bipartite-const}

\begin{claimproof}[Proof of \eqref{clm:bipartite-const}]
Suppose first that $H$ is an induced subgraph of $G[\bigcup_{i \in [r]} S_i]$ with $w(V(H))=m$ and which admits a $(d, r)$-coloring $\phi$ satisfying both \ref{itm:property1} and \ref{itm:property2}. By \ref{itm:property2}, for every $i \in [r]$, each vertex $v \in V(H)$ with $\phi(v) = i$ belongs to $S_i$. But since each connected component of $G[S_i]$ is a clique, vertices in $H$ assigned color $i$ under $\phi$ belong to distinct connected components of $G[S_i]$. For every $i \in [r]$, let $M_i \subseteq [c_i]$ be the set of indices $j$ such that $S_i^j \cap \{v \in V(H) \colon \phi(v) = i\} \neq \varnothing$ (note that, by the previous remark, this intersection has size exactly one). Moreover, for every $j \in M_i$, let $v_i^j$ be the unique vertex in $S_i^j \cap \{v \in V(H) \colon \phi(v) = i\}$. It is not difficult to see that the set of edges
$$M = \bigcup_{i \in [r]} \bigcup_{j \in M_i} \{y_{v_i^j}x_{S_i^j}\}$$
is a matching in $B$ of weight $m$. 

Conversely, suppose that $B$ has a matching $M$ of weight $m$. Let $V \subseteq \bigcup_{i \in [r]} S_i$ be the set of vertices $v$ such that $y_v$ is incident to an edge in $M$. For every $i \in [r]$, let $V_i = \{v \in V \colon \exists j \in [c_i], y_vx_{S_i^j} \in M\}$. Note that, by construction of $B$, for every $v \in V$ there exists exactly one index $i \in [r]$ such that $v \in V_i$, and so $(V_1, \ldots, V_r)$ is a partition of $V$. Furthermore, for any $i \in [r]$ and any two distinct vertices $u,v \in V_i$, there exist two distinct indices $j,p \in [c_i]$ such that $y_vx_{S_i^j},y_ux_{S_i^p} \in M$; in particular, $u$ and $v$ belong to distinct connected components of $G[S_i]$ and so $d_{G}(u,v) \geq d$. It follows that the map $\phi\colon V \to [r]$ defined by $\phi(v) = i$ for every $v \in V_i$ and $i \in [r]$ is a $(d,r)$-coloring of $G[V]$. Moreover, for every $i \in [r]$, we have that $V_i = \{v \in V \colon \phi(v) = i\} \subseteq S_i$ by construction, which implies that for every $v \in V_i$, the color $i$ belongs to $L(v)$ thanks to the preprocessing step. Thus, the induced subgraph $G[V]$ of $G[\bigcup_{i \in [r]} S_i]$ has a $(d,r)$-coloring $\phi$ satisfying both \ref{itm:property1} and \ref{itm:property2}. Moreover, $w(V) = \sum_{i \in [r]} w(V_i) = w(M) = m$. 
\end{claimproof}

Now, by (\autoref{clm:bipartite-const}), our problem reduces to finding a maximum-weight matching in the bipartite graph $B$, which can be done in $O(|V(B)|^{3})$-time using the Hungarian method (see, e.g., \cite[Theorem~17.3]{Sch}. Since $|V(B)| \leq (r+1)|V(G)|$, this completes the proof of \autoref{lem:good-graphs}.
\end{proof}

We can finally prove \Cref{thm:main}, which we restate for convenience. 

\main*

\begin{proof}
Given a $kP_3$-free graph $G$ and an $r$-list assignment $L$ of $G$, we show that the following algorithm outputs, in polynomial time, an induced subgraph $H$ of $G$ admitting a coloring that respects $L$ and such that $H$ is of maximum weight among all such subgraphs of $G$. 

\medskip

\begin{enumerate}[label=\textbf{Step \arabic*.},ref=\arabic*,leftmargin=*]
	\item Compute an amiable family $\mathcal{S}$ of $G$ of size $|V(G)|^{O(k)}$.
	\item For each $r$-tuple $(S_1,\ldots,S_r) \in \mathcal{S}^r$, find a maximum-weight induced subgraph $H$ of $G[\bigcup_{i \in [r]} S_i]$ which admits a coloring $\phi\colon V(H) \to [r]$ satisfying \ref{itm:property1} and \ref{itm:property2} of \autoref{lem:good-graphs}.
	\item Among all induced subgraphs computed in Step $2$, output one of maximum weight.
\end{enumerate}

\medskip

We first show that the algorithm is correct. To see this, observe that if $H$ is an induced subgraph of $G$ and $\phi\colon V(H) \to [r]$ is a coloring of $H$ such that $\phi(v) \in L(v)$ for every $v \in V(H)$, then each color class $C_i = \{v \in V(H) \colon \phi(v) = i\}$ is an independent set of $G$. Hence, since $\mathcal{S}$ is an amiable family of $G$, for every $i \in [r]$ there exists $S_i \in \mathcal{S}$ such that $C_i \subseteq S_i$. The correctness of the algorithm then follows from \autoref{lem:amiablefam} and \autoref{lem:good-graphs}, where we set $d=2$.

Consider now the running time. By \autoref{lem:amiablefam}, an amiable family $\mathcal{S}$ of $G$ of size $|V(G)|^{O(k)}$ as in Step $1$ can be computed in time $|V(G)|^{O(k)}$. Observe now that, in Step $2$, the number of $r$-tuples to consider is at most $|\mathcal{S}|^{r} = |V(G)|^{O(rk)}$. Moreover, for each fixed $r$-tuple, a maximum-weight induced subgraph $H$ of $G[\bigcup_{i \in [r]} S_i]$ with the desired properties can be computed in time $O(((r+1)|V(G)|)^{3})$ thanks to \Cref{lem:good-graphs}. Therefore, the total running time is $|V(G)|^{O(rk)}$.
\end{proof}


\section{Distance-$d$ amiable families: The proofs of \Cref{lem:dist-amiablefam} and \Cref{thm:distanceCol}}\label{sec:distancedamiable}

In this section we first prove \Cref{lem:dist-amiablefam}, which will then be used for the proof of \Cref{thm:distanceCol}. To make our inductive proof of \Cref{lem:dist-amiablefam} work, we show in fact something more general, which requires the following definition. 

Given a graph $G$ and a subset $F \subseteq V(G)$, a subset $S \subseteq V(G)$ is \emph{$F$-avoiding} if $S \cap F = \varnothing$. By extension, a family $\mathcal{S} \subseteq 2^{V(G)}$ is \textit{$F$-avoiding} if each member of $\mathcal{S}$ is $F$-avoiding. For a graph $G$, a family $\mathcal{S} \subseteq 2^{V(G)}$ of subsets of $V(G)$ is an \textit{$F$-avoiding distance-$d$ amiable family} if it satisfies the following properties:
\begin{itemize}
    \item $\mathcal{S}$ is $F$-avoiding;
    \item Each member of $\mathcal{S}$ induces a $P_3$-free subgraph in $G$;
    \item For each $S \in \mathcal{S}$, the connected components of $G[S]$ are pairwise at distance at least $d$ in $G$;
    \item Each (inclusion-wise) maximal $F$-avoiding distance-$d$ independent set of $G$ is contained in some member of $\mathcal{S}$.
\end{itemize}

Note that a distance-$d$ amiable family of $G$ is nothing but an $F$-avoiding distance-$d$ amiable family of $G$ for $F = \varnothing$. 
As we shall see, our proof of \Cref{lem:dist-amiablefam} in fact shows that, for every $kP_3$-free graph $G$ and every $F \subseteq V(G)$, the graph $G$ admits an $F$-avoiding distance-$d$ amiable family of size $|V(G)|^{O(k)}$ and which can be computed in time $|V(G)|^{O(k)}$. 

Before we continue to the proof of \Cref{lem:dist-amiablefam}, let us discuss why our approach fails for $3 \leq d \leq 5$ (recall however that the failure for $d\in\{3,5\}$ is to be expected given the hardenss results in \Cref{eto}).
The algorithm for computing a distance-$d$ amiable family is, in essence, similar to the one for computing an amiable family (so $d = 2$): it enumerates all $P_3$'s $uvw$ in the graph, ``guesses'' which vertex in the path is in the independent set and recursively calls the algorithm on (roughly) the anti-neighborhood of $\{u,v,w\}$, which is $(k-1)P_3$-free. The main difference (and difficulty) is in ensuring that the family computed recursively satisfies the distance requirement: 
the connected components of each member in this family could in principle be closer in the original graph.
Consider, for instance, the case where $u$ is picked in the independent set, and there are four connected components $Q_1, Q_2, P_1, P_2$ in a member of the recursively computed family with shortest paths between $Q_1$ and $Q_2$, and $P_1$ and $P_2$ as shown in \Cref{fig:recursive}. 
Then, for $d \leq 5$, $Q_1$ and $Q_2$ can end up at a distance less than $d$, while $P_1$ and $P_2$ are at a distance at least $d$. Since there is a priori no way of ``distinguishing'' these two cases, we have to take $d$ ``large enough''. The precise argument appears in the proof of~\Cref{lem:dist-amiablefam}, Claim (\ref{clm:damiabilityP3}).

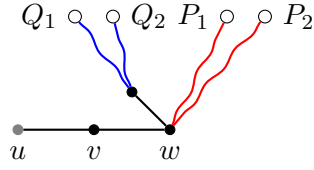
\begin{figure}
\captionsetup{width=\linewidth}
\begin{tikzpicture}
    \node[circg,label=below:{\small $u$}] (u) at (0,0) {};
    \node[circ,label=below:{\small $v$}] at (1,0) {};
    \node[circ,label=below:{\small $w$}] (w) at (2,0) {};
    \node[circ] (a) at (1.5,.5) {};
    \draw[thick] (u) -- (w);
    \draw[thick] (w) -- (a);

    \node[circb,label=left:{\small $Q_1$}] (Q1) at (.75,1.5) {};
    \node[circb,label=right:{\small $Q_2$}] (Q2) at (1.25,1.5) {};
    \draw[nonedge,blue,thick] (Q1) -- (a);
    \draw[nonedge,blue,thick] (a) -- (Q2);

    \node[circb,label=left:{\small $P_1$}] (P1) at (2.75,1.5) {};
    \node[circb,label=right:{\small $P_2$}] (P2) at (3.25,1.5) {};
    \draw[nonedge,red,thick] (w) -- (P1);
    \draw[nonedge,red,thick] (P2) -- (w);
\end{tikzpicture}
\caption{The case $d \leq 5$ (paths in blue are of length $d-3$, those in red, $d-2$).}
\label{fig:recursive}
\end{figure}

We are finally ready to prove \Cref{lem:dist-amiablefam}, restated below for convenience. 

\distamiablefam*

\begin{proof} To prove the lemma, we provide, for every $d \geq 6$ and every $k \geq 1$, an algorithm $\Lambda_k^d$ which takes as input a $kP_3$-free graph $G$, together with an arbitrary ordering of $V(G)$, and a subset $F \subseteq V(G)$, and outputs an $F$-avoiding distance-$d$ amiable family $\Lambda_k^d(G,F)$ of $G$. The pseudo-code of the algorithm is given in \Cref{algo:Lambda_k}. Note that, if the input graph $G$ is the empty graph, $\Lambda_k^d$ correctly outputs $\Lambda_k^d(G) = \{\varnothing\}$. 

In the following, we prove three claims (\eqref{clm:Favoiding}, \eqref{clm:damiabilityP3} and \eqref{clm:damiabilityIS}) which altogether show that, for every $d \geq 6$ and every $k \geq 1$, if $G$ is a $kP_3$-free graph and $F \subseteq V(G)$, then the family $\Lambda_k^d(G,F)$ is indeed an $F$-avoiding distance-$d$ amiable family. We will then show that $\Lambda_k^d(G,F)$ has size $|V(G)|^{O(k)}$ and that the algorithm  $\Lambda_k^d$ has running time $|V(G)|^{O(k)}$. Taking $F = \varnothing$, will prove the lemma.

Given a graph $G$ on $n$ vertices and an ordering $v_1,\ldots,v_n$ of $V(G)$, we let $G_i = G[\{v_1,\ldots,v_i\}]$. Moreover, an induced path in $G$ with vertex set $\{x_1, x_2, \ldots, x_{\ell}\}$ and edge set $\{x_1x_2, x_2x_3, \ldots x_{\ell-1}x_{\ell}\}$ is denoted by listing its vertices in the natural order $x_1x_2\cdots x_{\ell}$.

\begin{algorithm}{}
\caption{$\Lambda_k^d$}
\label{algo:Lambda_k}
\begin{algorithmic}[1]
\Require A $kP_3$-free graph $G$ with an arbitrary ordering $v_1, \ldots, v_n$ of $V(G)$, and a subset $F \subseteq V(G)$. 
\Ensure An $F$-avoiding distance-$d$ amiable family of $G$.
\State Initialize $\mathcal{S} = \{\varnothing\}$.
\For{$i = 1, \dots, n$}\label{mainloop}

\If{$v_i \notin F$}

\For{every member $S \in \mathcal{S}$}
\If{$G[S\cup \{v_i\}]$ is $P_3$-free and its connected components are pairwise at distance at least $d$ in $G$}
\State\label{step2.1} Set $S = S \cup \{v_i\}$.
\EndIf
\EndFor

\For{every induced $P_3$ $uvw$ in $G_i$ such that $u \notin F$} 
\State Compute the family $\mathcal{C} := \Lambda_{k-1}^d(G[N_G^{\geq 4}(u)], (F \cap N_G^{\geq 4}(u)) \cup (N_G^{\geq 4}(u) \cap N_G^{\leq d-1}(u)))$. 
\For{every $C \in \mathcal{C}$} 
\State\label{step2.2.1} Set $\mathcal{S} = \mathcal{S} \cup \{C \cup \{u\}\}$.
\EndFor
\EndFor

\For{every induced $P_3$ $uvw$ in $G_i$ such that $v \notin F$}
\State Compute the family $\mathcal{C} := \Lambda_{k-1}^d(G[N_G^{\geq 4}(v)], (F \cap N_G^{\geq 4}(v)) \cup (N_G^{\geq 4}(v) \cap N_G^{\leq d-1}(v)))$. 
\For{every $C \in \mathcal{C}$} 
\State\label{step2.2.2} Set $\mathcal{S} = \mathcal{S} \cup \{C \cup \{v\}\}$.
\EndFor
\EndFor

\EndIf

\EndFor
\State \Return $\Lambda_k^d(G,F) := \mathcal{S}$.
\end{algorithmic}
\end{algorithm}

\clm{For every $d \geq 6$ and every $k \geq 1$, if $G$ is a $kP_3$-free graph and $F \subseteq V(G)$, then $\Lambda_k^d(G,F)$ is $F$-avoiding.}\label{clm:Favoiding}

\begin{claimproof}[Proof of \eqref{clm:Favoiding}]
For fixed $d \geq 6$, we proceed by induction on $k$. The statement holds for $k=1$, since for every $P_3$-free graph $G$ and every $F \subseteq V(G)$, $\Lambda_1^d(G) = \{V(G) \setminus F\}$. Suppose now that $k > 1$ and that the statement holds for $k-1$. Let $G$ be an $n$-vertex $kP_3$-free graph and let $F \subseteq V(G)$ as in input. For every $i \in [n]$, let $\mathcal{S}^i$ be the state of the family $\mathcal{S}$ at the end of the $i$-th iteration of the main loop. We show by induction on $i$ that, for every $S \in \mathcal{S}^i$, $S \cap F = \varnothing$. The base case $i=1$ trivially holds, since either $v_1 \in F$ in which case $\mathcal{S}^1 = \{\varnothing\}$, or $v_1 \notin F$ in which case $\mathcal{S}^1 = \{\{v_1\}\}$. Thus, suppose that $i >1$ and that the statement holds for $i-1$.

Consider $S \in \mathcal{S}^i$. If $S \in \mathcal{S}^{i-1}$ then, by the induction hypothesis, $S \cap F = \varnothing$. Thus, we may assume that $S \notin \mathcal{S}^{i-1}$. This implies that $S$ is added to $\mathcal{S}^i$ in one of the three inner loops during the $i$-th iteration of the main loop. Suppose first that $S$ is added in line \ref{step2.1} as an extension of a set $Q \in \mathcal{S}^{i-1}$. Hence, by construction, $v_i \notin F$ and, by the induction hypothesis, $Q \cap F = \varnothing$, which implies that $S \cap F = \varnothing$. Suppose next that $S$ is added in line \ref{step2.2.1}. Hence, there exist an induced $P_3$ $uvw$ in $G_i$ such that $u \notin F$ and a set $C \in \Lambda_{k-1}^d(G[N_G^{\geq 4}(u)], (F \cap N_G^{\geq 4}(u)) \cup (N_G^{\geq 4}(u) \cap N_G^{\leq d-1}(u)))$ such that $S = C \cup \{u\}$. Observe now that since $G$ is $kP_3$-free and $(N_G(v) \cup N_G(w)) \cap N_G^{\geq 4}(u) = \varnothing$, the graph $G[N_G^{\geq 4}(u)]$ is $(k-1)P_3$-free, and thus, by the induction hypothesis on $k-1$, each member of $\Lambda_{k-1}^d(G[N_G^{\geq 4}(u)], (F \cap N_G^{\geq 4}(u)) \cup (N_G^{\geq 4}(u) \cap N_G^{\leq d-1}(u)))$ has empty intersection with $(F \cap N_G^{\geq 4}(u)) \cup (N_G^{\geq 4}(u) \cap N_G^{\leq d-1}(u))$. Since $C \subseteq N_G^{\geq 4}(u)$, it follows that $C \cap F = C \cap (F \cap N_G^{\geq 4}(u)) = \varnothing$ and so $S \cap F = \varnothing$. We conclude similarly if $S$ is added in line \ref{step2.2.2}.
\end{claimproof}

\clm{For every $d \geq 6$ and $k \geq 1$, if $G$ is a $kP_3$-free graph, $F \subseteq V(G)$, and $S \in \Lambda_k^d(G,F)$, then $G[S]$ is $P_3$-free and its connected components are pairwise at distance at least $d$ in $G$.}\label{clm:damiabilityP3}

\begin{claimproof}[Proof of \eqref{clm:damiabilityP3}]
For fixed $d \geq 6$, we proceed by induction on $k$. The statement holds for $k=1$, since for every $P_3$-free graph $G$ and every $F \subseteq V(G)$, $\Lambda_1^d(G) = \{V(G) \setminus F\}$. Suppose now that $k > 1$ and that the statement holds for $k-1$. Let $G$ be an $n$-vertex $kP_3$-free graph and let $F \subseteq V(G)$ as in input. For every $i \in [n]$, let $\mathcal{S}^i$ be the state of the family $\mathcal{S}$ at the end of the $i$-th iteration of the main loop. We show by induction on $i$ that each member of $\mathcal{S}^i$ induces a $P_3$-free subgraph of $G$ whose connected components are pairwise at distance at least $d$ in $G$. The base case $i = 1$ trivially holds, since either $v_1 \in F$ in which case $\mathcal{S}^1 = \{\varnothing\}$, or $v_1 \notin F$ in which case $\mathcal{S}^1 = \{\{v_1\}\}$. Thus, suppose that $i > 1$ and that the statement holds for $i-1$.

Consider $S \in \mathcal{S}^i$. If $S \in \mathcal{S}^{i-1}$ then, by the induction hypothesis, $G[S]$ is $P_3$-free and its connected components are pairwise at distance at least $d$ in $G$. Thus, we may assume that $S \notin \mathcal{S}^{i-1}$. This implies that $S$ is added to $\mathcal{S}^i$ in one of the three inner loops during the $i$-th iteration of the main loop. If $S$ is added in line \ref{step2.1} as an extension of a member of $\mathcal{S}^{i-1}$, then the statement holds by construction. Suppose next that $S$ is added to $\mathcal{S}^i$ in line \ref{step2.2.1}. Hence, there exist an induced $P_3$ $uvw$ in $G_i$ such that $u \notin F$ and a set $C \in \Lambda_{k-1}^d(G[N_G^{\geq 4}(u)], (F \cap N_G^{\geq 4}(u)) \cup (N_G^{\geq 4}(u) \cap N_G^{\leq d-1}(u)))$ such that $S = C \cup \{u\}$. Observe now that since $G$ is $kP_3$-free and $(N_G(v) \cup N_G(w)) \cap N_G^{\geq 4}(u) = \varnothing$, the graph $G[N_G^{\geq 4}(u)]$ is $(k-1)P_3$-free, and thus, by the induction hypothesis on $k-1$, every member of $\Lambda_{k-1}^d(G[N_G^{\geq 4}(u)], (F \cap N_G^{\geq 4}(u)) \cup (N_G^{\geq 4}(u) \cap N_G^{\leq d-1}(u)))$ induces a $P_3$-free subgraph of $G$ whose connected components are pairwise at distance at least $d$ in $G[N_G^{\geq 4}(u)]$. We claim that then, the connected components of $G[C]$ are in fact pairwise at distance at least $d$ in $G$. Suppose, to the contrary, that there exist two connected components $Q_1$ and $Q_2$ of $G[C]$ such that $d_G(Q_1,Q_2) \leq d-1$. 
Since the distance in $G[N_G^{\geq 4}(u)]$ between $Q_1$ and $Q_2$ is at least $d$, every shortest path in $G$ from $Q_1$ to $Q_2$ contains at least one vertex of $V(G) \setminus N_G^{\geq 4}(u)$; let $P$ be such a shortest path and let $x \in V(G) \setminus N_G^{\geq 4}(u)$ be an arbitrary vertex on $P$. Clearly, $d_G(u,x) \leq 3$. Now, by \eqref{clm:Favoiding}, $C \cap ((F \cap N_G^{\geq 4}(u)) \cup (N_G^{\geq 4}(u) \cap N_G^{\leq d-1}(u))) = \varnothing$ which implies, in particular, that $C \subseteq N_G^{\geq d}(u)$. It follows that for $j \in  [2]$,
\[
d \leq d_G(u,Q_j) \leq d_G(u,x) + d_G(x,Q_j) \leq 3 + d_G(x,Q_j)
\]
and thus,
\[
d_G(Q_1,Q_2) = d_G(Q_1,x) + d_G(x,Q_2) \geq 2(d-3).
\]
Since, by assumption, $d_G(Q_1,Q_2) \leq d-1$, we conclude that $2(d-3) \leq d-1$, that is, $d \leq 5$, a contradiction. Thus, the connected components of $G[C]$ are pairwise at distance at least $d$ in $G$. Since $C \subseteq N_G^{\geq d}(u)$, as previously observed, and $G[C]$ is $P_3$-free, it follows that $S = C \cup \{u\}$ induces a $P_3$-free graph whose connected components are pairwise at distance at least $d$ in $G$. We conclude similarly in the case where $S$ is added to $\mathcal{S}^i$ in line \ref{step2.2.2}.
\end{claimproof}

\clm{For every $d \geq 6$ and $k \geq 1$, if $G$ is a $kP_3$-free graph and $F \subseteq V(G)$, then every $F$-avoiding distance-$d$ independent set of $G$ is contained in a member of $\Lambda_k^d(G,F)$.}\label{clm:damiabilityIS}

\begin{claimproof}[Proof of \eqref{clm:damiabilityIS}]
For fixed $d \geq 6$, we proceed by induction on $k$. The statement clearly holds for $k = 1$, since for every $P_3$-free graph $G$ and every $F \subseteq V(G)$, $\Lambda_1^d(G,F) = \{V(G) \setminus F\}$. Suppose now that $k > 1$ and that the statement holds for $k-1$. Let $G$ be a $kP_3$-free graph and let $F \subseteq V(G)$ as in input. For every $i \in [n]$, let $\mathcal{S}^i$ be the state of the family $\mathcal{S}$ at the end of the $i$-th iteration of the main loop. Furthermore, given a distance-$d$ independent set $I$ of $G$ such that $I \subseteq V(G_i)$, we say that $I$ is \emph{$G_i$-compatible} if for every $u\in I$ and every connected component $Q$ of $G_i$ such that $u$ is connected to $Q$ in $G$, there exists a shortest $u, Q$-path $P_{u,Q}$ in $G$ such that $N_G^{\leq 2}(u) \cap V(P_{u,Q}) \subseteq V(G_i)$. We now show, by induction on $i$, that for every $G_i$-compatible $F$-avoiding distance-$d$ independent set $I$ of $G$ such that $I \subseteq V(G_i)$, there exists $S \in \mathcal{S}^i$ such that $I \subseteq S$. Note that this is enough to prove our statement for $k$, since any $F$-avoiding distance-$d$ independent set of $G = G_n$ is surely $G_n$-compatible. 

The base case $i = 1$ trivially holds, since either $v_1 \in F$ and $\mathcal{S}^1 = \{\varnothing\}$ in which case $\varnothing$ is the only $G_1$-compatible $F$-avoiding distance-$d$ independent of $G_1$, or $v_1 \notin F$ and $\mathcal{S}^1 = \{\{v_1\}\}$ in which case $\{v_1\}$ is the only maximal $G_1$-compatible $F$-avoiding distance-$d$ independent set of $G_1$. Thus, suppose that $i > 1$ and that the statement holds for $i-1$.

Consider a $G_i$-compatible $F$-avoiding distance-$d$ independent set $I$ of $G$ such that $I \subseteq V(G_i)$. If $I \subseteq V(G_{i-1})$ and $I$ is $G_{i-1}$-compatible then, by the induction hypothesis, there exists $S \in \mathcal{S}^{i-1}$ such that $I \subseteq S$. Observe now that, by construction, there then exists $S' \in \mathcal{S}^i$ such that $S \subseteq S'$ and hence $I \subseteq S'$. Thus, assume that this does not hold. This implies that either $I \not\subseteq V(G_{i-1})$, or $I \subseteq V(G_{i-1})$ but $I$ is not $G_{i-1}$-compatible. We distinguish these two cases and argue that we can always find $S \in \mathcal{S}^i$ such that $I \subseteq S$.

\textbf{Case 1.} $I \not\subseteq V(G_{i-1})$. 

\noindent Hence, $v_i \in I$. Note that since $I$ is $F$-avoiding, $v_i \notin F$. We claim that then $I \setminus \{v_i\}$ is $G_{i-1}$-compatible. Indeed, since $I$ is $G_i$-compatible, for every $u \in I \setminus \{v_i\}$ and every connected component $Q$ of $G_i$ such that $u$ is connected to $Q$ in $G$, there exists a shortest $u, Q$-path $P_{u,Q}$ in $G$ such that $N_G^{\leq 2}(u) \cap V(P_{u,Q}) \subseteq V(G_i)$. But $v_i$ is at distance at least $d \geq 6$ from $u$ in $G$ and so $v_i \notin (N_G^{\leq 2}(u) \cap V(P_{u,Q}))$, that is, $N_G^{\leq 2}(u) \cap V(P_{u,Q}) \subseteq V(G_{i-1})$ which proves our claim. Now, by the induction hypothesis, there exists $S \in \mathcal{S}^{i-1}$ such that $I \setminus \{v_i\} \subseteq S$. If $G[S \cup \{v_i\}]$ is $P_3$-free and its connected components are pairwise at distance at least $d$ in $G$, then $S \cup \{v_i\} \in \mathcal{S}^i$ by construction (cf. line \ref{step2.1}) and $I \subseteq S \cup \{v_i\}$. Thus, we may assume that this does not hold. Hence, either $G[S \cup \{v_i\}]$ contains at least one induced $P_3$, or $G[S \cup \{v_i\}]$ is $P_3$-free but at least two of its connected components are at distance strictly less than $d$ in $G$.

\textbf{Case 1.1.} $G[S \cup \{v_i\}]$ contains at least one induced $P_3$. 

\noindent By \eqref{clm:damiabilityP3}, $G[S]$ is $P_3$-free and thus, any induced $P_3$ in $G[S \cup \{v_i\}]$ must contain $v_i$. Similarly to the proof of \eqref{amiablefam:IS}, we distinguish two cases. Suppose first that there exist a connected component $Q$ of $G[S]$ and a vertex $u \in Q$ such that $v_i$ is adjacent to $u$ but not complete to $Q$, say $w \in Q$ is nonadjacent to $v_i$. Note that at some point during the $i$-th iteration of the main loop, the induced $P_3$ $v_iuw$ is considered in the second inner loop, since $v_i \notin F$. Furthermore, since $G$ is $kP_3$-free and $(N_G(u) \cup N_G(w)) \cap N_G^{\geq 4}(v_i) =\varnothing$, the graph $G[N_G^{\geq 4}(v_i)]$ is $(k-1)P_3$-free. Thus, by the induction hypothesis on $k-1$ and since $I \setminus \{v_i\} \subseteq N_G^{\geq d}(v_i) \setminus F$, there exists $C \in \Lambda_{k-1}^d(G[N_G^{\geq 4}(v_i)],(F \cap N_G^{\geq 4}(v_i)) \cup (N_G^{\geq 4}(v_i) \cap N_G^{\leq d-1}(v_i)))$ such that $I \setminus \{v_i\} \subseteq C$. But then, $C \cup \{v_i\} \in \mathcal{S}^i$ and $I \subseteq C \cup \{v_i\}$. The case where, for every connected component $Q$ of $G[S]$ the vertex $v_i$ is either complete or anticomplete to $Q$, is handled similarly.

\textbf{Case 1.2.} $G[S \cup \{v_i\}]$ is $P_3$-free but at least two of its connected components are at distance strictly less than $d$ in $G$.

\noindent By \eqref{clm:damiabilityP3}, the connected components of $G[S]$ are pairwise at distance at least $d$ in $G$ and so there must exist a connected component $Q$ of $G[S \cup \{v_i\}]$ at distance strictly less than $d$ in $G$ to the connected component of $G[S \cup \{v_i\}]$ containing $v_i$. Now, since $v_i$ is connected to $Q$ in $G$ and $v_i \in I$, the assumption that $I$ is $G_i$-compatible implies that there exists a shortest $v_i,Q$-path $P_{v_i,Q}$ in $G$ such that $N_G^{\leq 2}(v_i) \cap V(P_{v_i,Q}) \subseteq V(G_i)$. Let $u_1,u_2 \in V(G_i)$ be the two vertices on $P_{v_i,Q}$ at distance one and two, respectively, from $v_i$. Note that at some point during the $i$-th iteration of the main loop, the induced $P_3$ $v_iu_1u_2$ is considered in the second inner loop, since $v_i \notin F$. Moreover, since $G$ is $kP_3$-free and $(N_G(u_1) \cup N_G(u_2)) \cap N_G^{\geq 4}(v_i) = \varnothing$, the graph $G[N_G^{\geq 4}(v_i)]$ is $(k-1)P_3$-free. Thus, by the induction hypothesis on $k-1$ and since $I \subseteq N_G^{\geq d}(v_i) \setminus F$, there exists $C \in \Lambda_{k-1}^d(G[N_G^{\geq 4}(v_i)],(F \cap N_G^{\geq 4}(v_i)) \cup (N_G^{\geq 4}(v_i) \cap N_G^{\leq d-1}(v_i)))$ such that $I \setminus \{v_i\} \subseteq C$. But then, $C \cup \{v_i\} \in \mathcal{S}^i$ and $I \subseteq C \cup \{v_i\}$.

\textbf{Case 2.} $I \subseteq V(G_{i-1})$ but $I$ is not $G_{i-1}$-compatible. 

\noindent Hence, there must exist $u \in I$ and a connected component $Q$ of $G_i$ to which $u$ is connected in $G$ such that $v_i \in N_G^{\leq 2}(u) \cap V(P_{u,Q})$, where $P_{u,Q}$ is a shortest path in $G$ from $u$ to $Q$ given by the $G_i$-compatibility of $I$. Let $u_1,u_2 \in V(G_i)$ be the vertices on $P_{u,Q}$ at distance one and two, respectively, from $u$. Note that at some point during the $i$-th iteration of the main loop, the induced $P_3$ $uu_1u_2$ is considered in the second inner loop, since $v_i \notin F$. Moreover, since $G$ is $kP_3$-free and $(N_G(u_1) \cup N_G(u_2)) \cap N_G^{\geq 4}(u) = \varnothing$, the graph $G[N_G^{\geq 4}(u)]$ is $(k-1)P_3$-free. Thus, by the induction hypothesis on $k-1$ and since $I \subseteq N_G^{\geq d}(u) \setminus F$, there exists $C \in \Lambda_{k-1}^d(G[N_G^{\geq 4}(u)],(F \cap N_G^{\geq 4}(u)) \cup (N_G^{\geq 4}(u) \cap N_G^{\leq d-1}(u)))$ such that $I \setminus \{u\} \subseteq C$. But then, $C \cup \{u\} \in \mathcal{S}^i$ and $I \subseteq C \cup \{u\}$.

Since in all cases we found $S \in \mathcal{S}^i$ such that $I \subseteq S$, this concludes the proof of \eqref{clm:damiabilityIS}. 
\end{claimproof}

It follows now from \eqref{clm:Favoiding}, \eqref{clm:damiabilityP3} and \eqref{clm:damiabilityIS} that, for every $d \geq 6$ and every $k \geq 1$, if $G$ is a $kP_3$-free graph and $F \subseteq V(G)$, then the family $\Lambda_k^d(G,F)$ is indeed an $F$-avoiding distance-$d$ amiable family. It remains to show that $\Lambda_k^d(G,F)$ has size at most $|V(G)|^{O(k)}$ and that the running time of the algorithm $\Lambda_k^d$ is $|V(G)|^{O(k)}$. To this end, let
$$f_d(n,k) = \max \{|\Lambda_k^d(G,F)|: |V(G)| \leq n, \ F \subseteq V(G) \text{ and } G \text{ is } kP_3\text{-free}\}.$$ Clearly, $f_d(n,1) = 1$ for every $n \in \mathbb{N}$. We claim that, for every $n \in \mathbb{N}$ and $k > 1$,
$f_d(n,k) \leq 2n^4\cdot f_d(n,k-1)$. Indeed, for every $n$-vertex $kP_3$-free graph $G$ and every $F \subseteq V(G)$, a member of $\Lambda_k^d(G,F)$ can only be created during an $i$-th iteration of the main loop (for some $i \in [n]$) in one of the inner loops from an induced $P_3$ of $G_i$ and some set resulting from a call to $\Lambda_{k-1}^d$. Since for each $i \in [n]$ there are at most $i^3$ such copies of $P_3$ in $G_i$, at most $2i^3 \cdot f_d(n,k-1)$ new members are added in the $i$-th iteration of the main loop. It follows that 
$f_d(n,k) \leq \sum_{i \in [n]} 2i^3 \cdot f_d(n,k-1) \leq 2n^4 \cdot f_d(n,k-1)$ and thus
$f_d(n,k) \leq 2^{k-1}n^{4(k-1)}$.

Similarly, for every $d \geq 6$ and every $n,k \in \mathbb{N}$, if $T_d(n,k)$ denotes the running time of the algorithm $\Lambda_k^d$ on an $n$-vertex $kP_3$-free graph, then clearly $T_d(n,1) = O(n)$. Furthermore, 
we obtain the following recurrence for $T_d(n,k)$, where we use the fact that checking if an $n$-vertex graph is $P_3$-free can be done in $O(n^3)$ time and determining if the connected components of an $n$-vertex graph are pairwise at distance at least $d$ can be done in $O(n^3)$ time (using an all-pair-shortest-path algorithm): \[T_d(n,k) \leq cn\cdot(f_d(n,k)\cdot n^3 + 2n^3\cdot(T_d(n,k-1) + f_d(n,k-1))) \leq 2cn^4 \cdot T_d(n,k-1) + O(n^{4k}),\] for some constant $c > 0$. We conclude that $T_d(n,k) \leq n^{O(k)}$. This completes the proof of \Cref{lem:dist-amiablefam}.
\end{proof}

We can finally prove \Cref{thm:distanceCol}, which we restate for convenience. The proof is similar to that of \Cref{thm:main}.

\distanceCol*

\begin{proof} Given a $kP_3$-free graph $G$ and an $r$-list assignment $L$ of $G$, we show that the following algorithm correctly determines, in polynomial time, whether $G$ admits a $(d, r)$-coloring that respects $L$. 

\medskip

\begin{enumerate}[label=\textbf{Step \arabic*.},ref=\arabic*,leftmargin=*]
	\item Compute a distance-$d$ amiable family $\mathcal{S}$ of $G$ of size $|V(G)|^{O(k)}$.
	\item For each $r$-tuple $(S_1,\ldots,S_r) \in \mathcal{S}^r$, find a maximum-size 
    induced subgraph $H$ of $G[\bigcup_{i \in [r]} S_i]$ which admits a $(d, r)$-coloring $\phi\colon V(H) \to [r]$ satisfying \ref{itm:property1} and \ref{itm:property2} of \autoref{lem:good-graphs}.
	\item Return \textsc{Yes}, if $V(G) \subseteq V(H)$, and \textsc{No} otherwise.
\end{enumerate}

\medskip

We first show that the algorithm is correct. To see this, simply observe that if $G$ admits a $(d, r)$-coloring $\phi\colon V(G) \to [r]$ that respects $L$, i.e., $\phi(v) \in L(v)$ for every $v \in V(G)$, then each color class $C_i = \{v \in V(G) \colon \phi(v) = i\}$ is a distance-$d$ independent set of $G$. Hence, since $\mathcal{S}$ is a distance-$d$ amiable family of $G$, for every $i \in [r]$ there exists $S_i \in \mathcal{S}$ such that $C_i \subseteq S_i$. The correctness of the algorithm then follows from \autoref{lem:dist-amiablefam} and \autoref{lem:good-graphs}.

Consider now the running time. By \autoref{lem:dist-amiablefam}, a distance-$d$ amiable family $\mathcal{S}$ of $G$ of size $|V(G)|^{O(k)}$ as in Step $1$ can be computed in time $|V(G)|^{O(k)}$. Observe now that, in Step $2$, the number of $r$-tuples to consider is at most $|\mathcal{S}|^{r} = |V(G)|^{O(rk)}$. Moreover, for each fixed $r$-tuple, a maximum-size
induced subgraph $H$ of $G[\bigcup_{i \in [r]} S_i]$ with the desired properties can be computed in time 
$O(((r+1)|V(G)|)^{3})$
thanks to \Cref{lem:good-graphs}. Therefore, the total running time is $|V(G)|^{O(rk)}$.
\end{proof}

We finish this section with the following remark about extending \Cref{thm:distanceCol} to the max-weight setting. 
Since distances may change in induced subgraphs, there are two ways to define \textsc{Max-Weight} \textsc{List $(d,r)$-Colorable Induced Subgraph}: one in terms of the distances in the original graph (our present algorithm could be tweaked to solve this version of the problem), one in terms of the distances in the induced subgraph (a version which arguably makes more sense than the first). However, in the latter case, the straightforward extension of our approach would essentially require computing a distance-$d$ amiable family for every induced subgraph, which is computationally prohibitive.


\section{Concluding remarks}

In this paper we made considerable progress toward complexity dichotomies on $H$-free graphs for \textsc{Max-Weight List $r$-Colorable Induced Subgraph}, its special cases \textsc{Max-Weight $r$-Colorable Induced Subgraph} and \textsc{Odd Cycle Transversal}, and the distance-$d$ generalizations of \textsc{Max-Weight Independent Set} and \textsc{List $r$-Coloring}. Closing the remaining gaps is a natural problem. In particular, we state two open problems which we think deserve further attention. 

We showed in \Cref{thm:main} that, for every $r,k \in \mathbb{N}$, \textsc{Max-Weight List $r$-Colorable Induced Subgraph} is polynomial-time solvable on $kP_3$-free graphs. Paired with known results from the literature (see \Cref{thm:literature-oct}), this leaves the case of $H=k_4P_4+k_3P_3+k_2P_2+k_1P_1$ (where $k_4 \geq 1$ and $k_4 + k_3 \geq 2$) as the \textit{only} remaining open case toward a complete complexity dichotomy for \textsc{Odd Cycle Transversal} on $H$-free graphs. In particular, we ask the following.

\begin{question}\label{q1}
   Is \textsc{Odd Cycle Transversal} polynomial-time solvable on $(P_4+P_3)$-free graphs?
\end{question}

We note that a positive answer to \Cref{q1} would require an approach different from that for $kP_3$-free graphs we provided. Indeed, since \textsc{Max-Weight List $5$-Colorable
Induced Subgraph} is $\mathsf{NP}$-hard on $(P_4 + P_2)$-free graphs \cite{couturier2015list}, we should not expect to efficiently find amiable families even in $(P_4+P_2)$-free graphs. 

We conclude with a question related to \textsc{Max-Weight Distance-$d$
Independent Set} on $kP_3$-free graphs. \Cref{thm:distanceIS} implies that the problem is polynomial-time solvable for every $d\geq 6$ and $k\in \mathbb{N}$. Moreover, Eto et al.~\cite{EGM14} showed that when $d=5$ and $d = 3$ the problem is $\mathsf{NP}$-hard on $2P_3$-free graphs and $2P_2$-free graphs, respectively. This leaves the case $d=4$ as the \textit{only} remaining open case toward a complexity dichotomy.

\begin{question}\label{q2}
	Let $k \in \mathbb{N}$. Is \textsc{Max-Weight Distance-$4$
Independent Set} polynomial-time solvable on $kP_3$-free graphs?
\end{question}

Showing that efficiently computable distance-$4$ amiable families exist for $kP_3$-free graphs would provide a positive answer to \Cref{q2}. However, such existence seems unlikely.


\bibliographystyle{abbrv}
\bibliography{ref}

\end{document}